\documentclass[12pt,reqno]{amsart}
\usepackage{amsfonts}
\usepackage{amssymb,color}
\usepackage{amssymb}

\usepackage[margin=3cm, a4paper]{geometry}

\def\normo#1{\left\|#1\right\|}

\def\brk#1{\left(#1\right)}

\def\norm#1{\|#1\|}

\def\wt#1{\widetilde{#1}}

\newcommand{\R}{{\mathbb R}}

\newcommand{\Z}{{\mathbb Z}}
\newcommand{\ft}{{\mathcal{F}}}

\newcommand{\les}{{\lesssim}}
\newcommand{\ges}{{\gtrsim}}

\def\norm#1{\|#1\|}
\def\normo#1{\left\|#1\right\|}

\def\wt#1{\widetilde{#1}}

\def\ve#1{\mathbf{#1}}

\newcommand{\F}{\mathcal{F}}

\newcommand{\cir}{\mathbb{S}}

\newcommand{\e}{\varepsilon}

\newcommand{\p}{\partial}

\newcommand{\Del}[1]{}

\numberwithin{equation}{section}

\newtheorem{thm}{Theorem}[section]

\newtheorem{lem}[thm]{Lemma}

\theoremstyle{remark}

\theoremstyle{remark}

\theoremstyle{definition}

\begin{document}
\subjclass[2010]{35Q55}
\keywords{Landau-Lifshitz-Gilbert equation, Schr\"odinger maps, Inviscid limit, Critical Besov Space}

\title[Landau-Lifshitz equation]{The inviscid limit for the Landau-Lifshitz-Gilbert equation in the critical Besov space}
\author[Z. Guo]{Zihua Guo}
\address{School of Mathematical Sciences, Monash University, VIC 3800, Australia \& LMAM, School of Mathematical Sciences, Peking
University, Beijing 100871, China}
\email{zihua.guo@monash.edu}

\author[C. Huang]{Chunyan Huang}
\address{School of Statistics and Mathematics, Central University of Finance and Economics, Beijing
100081, China}
\email{hcy@cufe.edu.cn}

\begin{abstract}
We prove that in dimensions three and higher the Landau-Lifshitz-Gilbert equation with small initial data in the critical Besov space is globally well-posed in a uniform way  with respect to the Gilbert damping parameter. Then we show that the global solution converges to that of the Schr\"odinger maps in the natural space as the Gilbert damping term vanishes. The proof is based on some studies on the derivative Ginzburg-Landau equations.
\end{abstract}

\maketitle

\section{Introduction}
In this paper we study the Cauchy problem for the Landau-Lifshitz-Gilbert (LLG) equation
\begin{align}\label{eq:Lanlif}
\p_t s=as\times \Delta s-\e s\times (s\times \Delta s), \quad s(x,0)=s_0(x),
\end{align}
where $s(x,t):\R^n\times \R\to \cir^2\subset \R^3$,  $\times$
denotes the wedge product in $\R^3$, $a\in \R$ and $\e>0$ is the
Gilbert damping parameter. The equation \eqref{eq:Lanlif} is one of the equations of  ferromagnetic spin chain, which
was proposed by Landau-Lifshitz \cite{LanLif} in studying the
dispersive theory of magnetisation of ferromagnets. Later on, such
equations were also found in the condensed matter physics.  The LLG equation has been studied extensively,  see
\cite{Lak,GuoDing} for an introduction on the equation.

Formally, if $a=0$, then  \eqref{eq:Lanlif} reduces to the heat flow
equations for harmonic maps
\begin{align}\label{eq:Harmap}
\p_t s=-\e s\times (s\times \Delta s), \quad s(x,0)=s_0(x),
\end{align}
and if $\e=0$, then \eqref{eq:Lanlif} reduces to the Schr\"odinger
maps
\begin{align}\label{eq:Schmap}
\p_t s=as\times \Delta s, \quad s(x,0)=s_0(x).
\end{align}
Both special cases have been objects of intense research. The purpose of this paper is to study the inviscid limit of \eqref{eq:Lanlif}, namely, to prove rigorously that the solutions of \eqref{eq:Lanlif} converges to the solutions of \eqref{eq:Schmap} as $\e\to 0$ under optimal conditions on the initial data.

The inviscid limit is an important topic in mathematical physics, and has been studied in various settings, e.g. for hyperbolic-dissipative equations such as Navier-Stokes equation to Euler equation (see \cite{HK} and references therein), for dispersive-dissipative equations such as KdV-Burgers equation to KdV equation (see \cite{Guo2}) and Ginzburg-Landau equation to Schr\"odinger equations (see \cite{Wang02, HW}). The LLG equation \eqref{eq:Lanlif} is an equation with both dispersive and dissipative effects. This can be seen from the stereographic projection transform. It was known that (see
\cite{LakNak}) let
\begin{align}
u=\mathbb{P}(s)=\frac{s_1+is_2}{1+s_3},
\end{align}
where $s=(s_1,s_2,s_3)$ is a solution to \eqref{eq:Lanlif}, then $u$ solves the following complex derivative Ginzburg-Landau type equation
\begin{align}\label{eq:dGL}
\begin{split}
(i\p_t+\Delta-i\e\Delta)u=&\frac{2a\bar
u}{1+|u|^2}\sum_{j=1}^n(\p_{x_j}u)^2-\frac{2i\e\bar
u}{1+|u|^2}\sum_{j=1}^n(\p_{x_j}u)^2\\
u(x,0)=&u_0.
\end{split}
\end{align}
On the other hand, the projection transform has an inverse
\begin{align}
\mathbb{P}^{-1}(u)=\brk{\frac{u+\bar u}{1+|u|^2},\frac{-i(u-\bar u)}{1+|u|^2},\frac{1-|u|^2}{1+|u|^2}}.
\end{align}
Therefore, \eqref{eq:Lanlif} is equivalent to \eqref{eq:dGL} assuming $\mathbb P$ and $\mathbb P^{-1}$ is well-defined, and we
will focus on \eqref{eq:dGL}.  The previous works \cite{IK,IK2,Bej,Bej2,BIK} on the Schr\"odinger maps ($\e=0$) were also based on this transform.  Note that \eqref{eq:dGL} is invariant
under the following scaling transform: for $\lambda>0$
\[u(x,t)\to u(\lambda x,\lambda^2 t), \quad u_0(x)\to u_0(\lambda x).\]
Thus the critical Besov space  is $\dot
B_{2,1}^{n/2}$ in the sense of scaling.

To study the inviscid limit, the crucial task is to obtain uniform well-posedness with respect to the inviscid parameter. Energy method was used in \cite{HK}. For dispersive-dissipative equations, one needs to exploit the dispersive effect uniformly. Strichartz estimates and energy estimates were used in \cite{Wang02} for Ginzburg-Landau equations, and Bourgain space was used in \cite{Guo2} for KdV-Burgers equations. In \cite{WW, HWG} the inviscid limit for the derivative Ginzburg-Landau equations were studied by using the Strichartz estimates, local smoothing estimates and maximal function estimates.  However, these results requires high regularities when applied to equation \eqref{eq:dGL}.
In this paper we will use Bourgain-type space and exploit the null structure that are inspired by the latest development for the Schr\"odinger maps ($\e=0$) (see \cite{BIKT, BIK,Bej,Bej2,IK,IK2,Guo}) to study \eqref{eq:dGL} with small initial data in the critical Besov space. In
\cite{Bej2} and \cite{IK2} it was proved independently that global
well-posedness for \eqref{eq:Schmap} holds for small data in the critical
Besov space. We will extend their results to \eqref{eq:Lanlif} uniformly with respect to $\e$. We exploit the Bourgain space in a different way from both \cite{Bej2} and \cite{IK2}. One of the novelties is the use of $X^{0,1}$-structure that results in many simplifications even for the Schr\"odinger maps. The presence of dissipative term brings many technical difficulties, e.g. the lack of symmetry in time and incompatibility with $X^{s,b}$ structure. We need to overcome these difficulties when extending the linear estimates for the Schr\"odinger equation to the Schr\"odinger-dissipative equation uniformly with respect to $\e$.

By scaling we may assume $a=\pm 1$. From now on, we assume $a=1$ since
the other case $a=-1$ is similar. For $Q\in \cir^{2}$, the space
$\dot{B}^{s}_Q$ is defined by
\[\dot{B}^{s}_Q=\dot{B}^{s}_Q(\R^n;\cir^{2})=\{f: \R^n\to \R^3; f-Q\in \dot{B}^{s}_{2,1},\, |f(x)|\equiv 1 \mbox{ a.e. in } \R^n\},\]
where $\dot{B}^{s}_{2,1}$ is the standard Besov space. It was known
the critical space is $\dot{B}^{n/2}_Q$. The main result of this
paper is

\begin{thm}\label{thm:Schmap}
Assume $n\geq 3$. The LLG equation \eqref{eq:Lanlif} is
globally well-posed for small data $s_0\in
\dot{B}^{n/2}_Q(\R^n;\cir^{2})$, $Q\in \cir^{2}$ in a uniform way
with respect to $\e\in (0,1]$. Moreover, for any $T>0$, the solution
converges to that of Schr\"odinger map \eqref{eq:Schmap} in
$C([-T,T]:\dot{B}^{n/2}_Q)$ as $\e \to 0$.
\end{thm}

As we consider the inviscid limit in the strongest topology (same space as the initial data), no convergence rate is expected. This can be seen from linear solutions for \eqref{eq:dGL}. However, if assuming initial data has higher regularity, one can have convergence rate $O(\e T)$ (see \eqref{eq:limitHigh} below).

\section{Definitions and Notations}

For $x, y\in \R$, $x\les y$ means that there exists a constant $C$
such that $x\leq Cy$, and $x\sim y$ means that $x\les y$ and $y\les
x$. We use $\ft(f)$, $\hat{f}$ to denote the space-time Fourier
transform of $f$, and $\ft_{x_i,t}f$ to denote the Fourier transform
with respect to $x_i,t$.

Let $\eta: \R\to [0, 1]$ be an even, non-negative, radially
decreasing smooth function such that: a) $\eta$ is compactly
supported in $\{\xi:|\xi|\leq 8/5\}$; b) $\eta\equiv 1$ for
$|\xi|\leq 5/4$. For $k\in \Z$ let
$\chi_k(\xi)=\eta(\xi/2^k)-\eta(\xi/2^{k-1})$, $\chi_{\leq
k}(\xi)=\eta(\xi/2^k)$, $\widetilde
\chi_k(\xi)=\sum_{l=-9n}^{9n}\chi_{k+l}(\xi)$, and then define the
Littlewood-Paley projectors $P_k, P_{\leq k}, P_{\geq k}$ on
$L^2(\R^n)$ by
\[\widehat{P_ku}(\xi)=\chi_k(|\xi|)\widehat{u}(\xi),\quad \widehat{P_{\leq
k}u}(\xi)=\chi_{\leq k}(|\xi|)\widehat{u}(\xi),\] and $P_{\geq
k}=I-P_{\leq k-1}$, $P_{[k_1,k_2]}=\sum_{j=k_1}^{k_2}P_j$. We also
define $\wt P_k u=\F^{-1} \widetilde\chi_k(|\xi|)\widehat{u}(\xi)$

Let $\cir^{n-1}$ be the unit sphere in $\R^n$. For $\ve e\in
\cir^{n-1}$, define $\widehat{P_{k,\ve
e}u}(\xi)=\widetilde\chi_{k}(|\xi\cdot \ve
e|)\chi_k(|\xi|)\widehat{u}(\xi)$. Since for $|\xi|\sim 2^k$ we have
\[\sum_{l=-5n}^{5n}\chi_{k+l}(\xi_1)+\cdots+\sum_{l=-5n}^{5n}\chi_{k+l}(\xi_n)\sim
1,\] then let
\[\beta_k^j(\xi)=\frac{\sum_{l=-5n}^{5n}\chi_{k+l}(\xi_j)}{\sum_{j=1}^n\sum_{l=-5n}^{5n}\chi_{k+l}(\xi_j)}\cdot \sum_{l=-1}^1\chi_{k+l}(|\xi|), \quad j=1,\cdots, n.\]
Define the operator $\Theta_k^j$ on $L^2(\R^n)$ by
$\widehat{\Theta_k^jf}(\xi)=\beta_k^j(\xi)\hat f(\xi)$, $1\leq j\leq n$. Let $\ve
e_1=(1,0,\cdots, 0), \cdots, \ve e_n=(0,\cdots, 0,1)$. Then we have
\begin{align}\label{eq:Pkdec}
P_k=\sum_{j=1}^nP_{k,\ve e_j}\Theta_k^j.
\end{align}
For any $k\in \Z$, we define the modulation projectors $Q_k, Q_{\leq k}, Q_{\geq k}$ on $L^2(\R^n\times \R)$ by
\[\widehat{Q_ku}(\xi,\tau)=\chi_k(\tau+|\xi|^2)\widehat{u}(\xi,\tau),\quad \widehat{Q_{\leq
k}u}(\xi,\tau)=\chi_{\leq k}(\tau+|\xi|^2)\widehat{u}(\xi,\tau),\]
and $Q_{\geq k}=I-Q_{\leq k-1}$,
$Q_{[k_1,k_2]}=\sum_{j=k_1}^{k_2}Q_j$.

For any $\ve e\in \cir^{n-1}$, we can decompose $\R^n=\lambda \ve e
\oplus H_{\ve e}$, where $H_{\ve e}$ is the hyperplane with normal vector
$\ve e$, endowed with the induced measure. For $1\leq p, q< \infty$,
we define $L_{\ve e}^{p,q}$ the anisotropic Lebesgue space by
\[\norm{f}_{L_{\ve e}^{p,q}}=\brk{\int_\R\brk{\int_{H_{\ve e}\times \R} |f(\lambda \ve e+y,t)|^qdydt}^{p/q}d\lambda}^{1/p}\]
with the usual definition if $p=\infty$ or $q=\infty$. We write
$L_{\ve e_j}^{p,q}=L_{x_j}^pL_{\bar{x_j},t}^q$. We use
$\dot{B}^s_{p,q}$ to denote the homogeneous Besov spaces on $\R^n$
which is the completion of the Schwartz functions under the norm
\[\norm{f}_{\dot{B}^s_{p,q}}=(\sum_{k\in \Z}2^{qsk}\norm{P_kf}_{L^p}^q)^{1/q}.\]

To exploit the null-structure we also need the Bourgain-type space associated to the Schr\"odinger equation.
In this paper we use the modulation-homogeneous version as in \cite{Bej2, Guo}. We define $X^{0,b,q}$ to be the completion of the space of Schwartz functions with the norm
\begin{align}\label{eq:Xbq}
\norm{f}_{X^{0,b,q}}=(\sum_{k\in \Z}2^{kbq}\norm{Q_k
f}_{L^2_{t,x}}^q)^{1/q}.
\end{align}
If $q=2$ we simply write $X^{0,b}=X^{0,b,2}$. By the Plancherel's
equality we have
$\norm{f}_{X^{0,1}}=\norm{(i\p_t+\Delta)f}_{L^2_{t,x}}$. Since
$X^{0,b,q}$  is not closed under conjugation, we also define the
space $\bar X^{0,b,q}$ by the norm $\norm{f}_{\bar
X^{0,b,q}}=\norm{\bar f}_{X^{0,b,q}}$, and similarly write $\bar
X^{0,b}=\bar X^{0,b,2}$. It's easy to see that $X^{0,b,q}$ function
is unique modulo solutions of the homogeneous Schr\"odinger
equation. For a more detailed description of the $X^{0,b,p}$ spaces
we refer the readers to \cite{Tataru98} and \cite{Tao01}. We use
$X_+^{0,b,p}$ to denote the space restricted to the interval
$[0,\infty)$:
\[\norm{f}_{X_+^{0,b,p}}=\inf_{\tilde f: \tilde f=f \mbox{ on } t\in [0,\infty)}\norm{\tilde f}_{X^{0,b,p}}.\]
In particular, we have
\begin{align}\label{eq:X01ext}
\norm{f}_{X_+^{0,1}}\sim \norm{\tilde f}_{X^{0,1}}
\end{align}
where $\tilde f=f(t,x)1_{t\geq 0}+f(-t,x)1_{t<0}$.

Let $L=\partial_t-i\Delta$ and $\bar
L=\partial_t+i\Delta$. We define the main
dyadic function space. If $f(x,t)\in L^2(\R^n\times \R_+)$ has
spatial frequency localized in $\{|\xi|\sim 2^k\}$, define
\begin{align*}
\norm{f}_{F_{k}}=&\norm{f}_{X_+^{0,1/2,\infty}}+\norm{f}_{L_t^\infty
L_x^2}+\norm{f}_{L_t^2 L_x^\frac{2n}{n-2}}\\
&+2^{-(n-1)k/2}\sup_{\ve{e}\in
\cir^{n-1}}\norm{f}_{L_{\ve{e}}^{2,\infty}}+2^{k/2}\sup_{|j-k|\leq
20}\sup_{\ve{e}\in
\cir^{n-1}}\norm{P_{j,\ve e} f}_{L_{\ve{e}}^{\infty,2}},\\
\norm{f}_{Y_{k}}=&\norm{f}_{L_t^\infty L_x^2}+\norm{f}_{L_t^2
L_x^\frac{2n}{n-2}}+2^{-(n-1)k/2}\sup_{\ve{e}\in
\cir^{n-1}}\norm{f}_{L_{\ve{e}}^{2,\infty}}\\
&+2^{-k}\inf_{f=f_1+f_2}(\norm{L f_1}_{L_{t,x}^2}+\norm{\bar L f_2}_{L_{t,x}^2}),\\
\norm{f}_{Z_k}=&2^{-k}\norm{L f}_{L^2_{t,x}}\\
\norm{f}_{N_{k}}=&\inf_{f=f_1+f_2+f_3}(\norm{f_1}_{L_t^1L_x^2}+2^{-k/2}\sup_{\ve{e}\in
\cir^{n-1}}\norm{f_2}_{L_{\ve{e}}^{1,2}}+\norm{f_3}_{X_+^{0,-1/2,1}})+2^{-k}\norm{f}_{L^2_{t,x}}.
\end{align*}
Then we define the space $F^s, Y^s, Z^s, N^{s}$ with the following norm
\begin{align*}
\norm{u}_{F^{s}}=\sum_{k\in \Z}2^{ks}\norm{P_k u}_{F_{k}},\quad& \norm{u}_{Y^{s}}=\sum_{k\in \Z}2^{ks}\norm{P_k
u}_{Y_{k}},\\
\norm{u}_{Z^{s}}=\sum_{k\in \Z}2^{ks}\norm{P_k u}_{Z_{k}},\quad&
\norm{u}_{N^{s}}=\sum_{k\in \Z}2^{ks}\norm{P_k
u}_{N_{k}}.
\end{align*}
Obviously, $F_k\cap Z_k\subset Y_k$, and thus $F^s\cap Z^s\subset Y^s$.
In the end of this section, we present a standard extension lemma (See Lemma 5.4 in \cite{book}) giving the relation between $X^{s,b}$ and other space-time norm.
\begin{lem}\label{lem:ext}
Let $k\in \Z$ and $B$ be a space-time norm satisfying with some $C(k)$
\[\norm{e^{it_0}e^{it\Delta}P_kf}_B\leq C(k)\norm{P_kf}_2\]
for any $t_0\in \R$ and $f\in L^2$.  Then
\[\norm{P_ku}_B\les C(k) \norm{P_ku}_{X^{0,1/2,1}}.\]
\end{lem}

\section{Uniform linear estimates}
In this section we prove some uniform linear estimates for the equation \eqref{eq:dGL}  with respect to the dissipative parameter. First we recall the known linear estimates for the Schr\"odinger equation, see \cite{KT} and \cite{IK2}.

\begin{lem}\label{lem:linearStr}
Assume $n\geq 3$. For any $k\in \Z$ we have
\begin{align}
\norm{e^{it\Delta}P_kf}_{L_t^2 L_x^\frac{2n}{n-2}\cap L_t^\infty L_x^2}\les& \norm{P_kf}_2,\\
\sup_{\ve{e}\in
\cir^{n-1}}\norm{e^{it\Delta}P_kf}_{L_{\ve{e}}^{2,\infty}}\les& 2^{\frac{(n-1)k}{2}}\norm{P_kf}_2,\\
\sup_{\ve{e}\in
\cir^{n-1}}\norm{e^{it\Delta}P_{k,\ve e}f}_{L_{\ve{e}}^{\infty,2}}\les& 2^{-\frac{k}{2}}\norm{P_kf}_2.
\end{align}
\end{lem}

\begin{lem}\label{lem:linearXsb}
Assume $n\geq 3$, $u,F$ solves the equation: for $\e>0$
\[u_t-i\Delta u-\e \Delta u=F(x,t), \quad u(x,0)=u_0.\]
Then for $b\in [1,\infty]$
\begin{align}
\norm{P_ku}_{X_+^{0,1/2,b}}\les&
\norm{P_k u_0}_{L^2}+\norm{P_kF}_{X_+^{0,-1/2,b}},\\
\norm{P_ku}_{Z_k}\les&\e^{1/2}\norm{P_ku_0}_{L^2}+
2^{-k}\norm{P_kF}_{L_{t,x}^2},
\end{align}
where the implicit constant is independent of $\e$.
\end{lem}
\begin{proof}
First we show the second inequality.  We have
\begin{align}
u=e^{it\Delta+\e t\Delta}u_0+\int_0^te^{i(t-s)\Delta+\e (t-s)\Delta}F(s)ds
\end{align}
and thus
\begin{align*}
\norm{P_ku}_{Z_k}\les& \e 2^k\norm{P_ku}_{L^2_{t,x}}+
2^{-k}\norm{P_kF}_{L_{t,x}^2}\\
\les&\e^{1/2}\norm{P_ku_0}_{L^2}+2^{-k}\norm{P_kF}_{L_{t,x}^2}.
\end{align*}
Now we show the first inequality. We only prove the case $b=\infty$ since the other cases are similar.
First we assume $F=0$. Then $u=e^{it\Delta+\e t\Delta}u_0$. Let
$\tilde u=e^{it\Delta+\e |t|\Delta}u_0$, then $\tilde u$ is an
extension of $u$. Then
\begin{align*}
\norm{P_ku}_{X_+^{0,1/2,\infty}}\les
&\sup_j 2^{j/2}\norm{\F_t(e^{-it|\xi|^2-\e|t|\cdot |\xi|^2})(\tau)\hat u_0(\xi)\chi_k(\xi)\chi_j(\tau+|\xi|^2)}_{L_{\tau,\xi}^2}\\
\les&\sup_j
2^{j/2}\norm{(\e|\xi|^2)^{-1}\bigg(1+\frac{|\tau+|\xi|^2|^2}{(\e|\xi|^2)^2}\bigg)^{-1}\hat
u_0(\xi)\chi_k(\xi)\chi_j(\tau+|\xi|^2)}_{L_{\tau,\xi}^2}\\
\les& \norm{P_ku_0}_2.
\end{align*}
Next we assume $u_0=0$. Fix an extension $\tilde F$ of $F$ such that
\[\norm{\tilde F}_{X^{0,-1/2,\infty}}\leq
2\norm{F}_{X_+^{0,-1/2,\infty}}.\] Then define $\tilde
u=\F^{-1}\frac{1}{\tau+|\xi|^2+i\e |\xi|^2}\F \tilde F$. We see
$\tilde u$ is an extension of $u$ and then
\begin{align*}
\norm{P_ku}_{X_+^{0,1/2,\infty}}\les& \norm{\tilde
u}_{X^{0,1/2,\infty}}\\
\les&\sup_j 2^{j/2}\norm{\chi_k(\xi)\chi_j(\tau+|\xi|^2)\frac{1}{\tau+|\xi|^2+i\e |\xi|^2}\F \tilde F}_{L_{\tau,\xi}^2}\\
\les&\norm{\tilde F}_{X^{0,-1/2,\infty}}\les
\norm{F}_{X_+^{0,-1/2,\infty}}.
\end{align*}
Thus we complete the proof.
\end{proof}

\begin{lem}\label{lem:max}
Let $n\geq 3$, $k\in \Z,\e\geq 0$. Assume $u,F$ solves the equation
\[u_t-i\Delta u-\e \Delta u=F(x,t), \quad u(x,0)=u_0.\]
Then for any $\ve e\in \cir^{n-1}$ we have
\begin{align}
\norm{P_{k}u}_{L_{\ve e}^{2,\infty}}\les&
2^{k(n-1)/2}\norm{u_0}_{L^2}+2^{k(n-2)/2}\sup_{\ve e\in
\cir^{n-1}}\norm{F}_{L_{\ve e}^{1,2}},\\
\norm{P_{k,\ve e}u}_{L_{\ve e}^{\infty,2}}\les&
2^{-k/2}\norm{u_0}_{L^2}+2^{-k}\sup_{\ve e\in
\cir^{n-1}}\norm{F}_{L_{\ve e}^{1,2}},
\end{align}
where the implicit constant is independent of $\e$.
\end{lem}

\begin{proof}
If $\e=0$, then the inequalities were proved in \cite{IK2}. By the
scaling and rotational invariance, we may assume $k=0$ and $\ve
e=(1,0,\cdots,0)$.  Then the second inequality follows from Proposition 2.5,
2.7 in \cite{HWG}.  We prove the first inequality by the following two steps.

{\bf Step 1:} $F=0$.

From the fact
\begin{align*}
|e^{\e t\Delta}f(\cdot,t)(x)|\leq& (\e t)^{-n/2}\int
e^{-\frac{|x-y|^2}{2\e t}}|f(y,t)|dy\\
\les& (\e t)^{-n/2}\int e^{-\frac{|x-y|^2}{2\e
t}}\norm{f(y,t)}_{L_t^\infty}dy
\end{align*}
we get
\begin{align*}
\norm{e^{it\Delta+\e
t\Delta}P_0u_0}_{L_{x_1}^2L_{\bar{x}_1,t}^{\infty}} \les
\norm{e^{it\Delta}P_0u_0}_{L_{x_1}^2L_{\bar{x}_1,t}^{\infty}}\les
\norm{u_0}_2.
\end{align*}

{\bf Step 2:} $u_0=0$.

Decompose $P_0u=U_1+\cdots U_n$ such that $\ft_xU_i$ is supported in
$\{|\xi|\sim 1: |\xi_i|\sim 1\}\times \R$. Thus it suffices to show
\begin{align}
\norm{U_i}_{L_{x_1}^2L_{\bar{x},t}^{\infty}}\les \norm{F}_{L_{\ve
e_i}^{1,2}}.
\end{align}
We only show the estimate for $U_1$. We still write $u=U_1$. We
assume $\ft_xF$ is supported in $\{|\xi|\sim 1: \xi_1\sim 1\}\times
\R$. We have
\begin{align*}
u(t,x)=&\int_{\R^{n+1}}\frac{e^{it\tau}e^{ix\xi}}{\tau+|\xi|^2+i\e
|\xi|^2}\widehat{F}(\xi,\tau)d\xi
d\tau\\
=&\int_{\R^{n+1}}\frac{e^{it\tau}e^{ix\xi}}{\tau+|\xi|^2+i\e
|\xi|^2}\widehat{F}(\xi,\tau)(1_{\{-\tau-|\bar \xi|^2\sim
1\}^c}+1_{-\tau-|\bar\xi|^2\sim 1, |\tau+|\xi|^2|\les \e}\\
&+1_{-\tau-|\bar\xi|^2\sim 1, |\tau+|\xi|^2|\gg \e})d\xi d\tau\\
:=&u_1+u_2+u_3.
\end{align*}
For $u_1$, we simply use the Plancherel equality and get
\[\norm{\Delta u_1}_{L^2}+\norm{\partial_t u_1}_2\leq \norm{F}_2,\]
and thus by Sobolev embedding and Bernstein's inequality we obtain
the desired estimate. For $u_2$, using the Lemma \ref{lem:ext}, Lemma \ref{lem:linearStr} and
Bernstein's inequality we get
\begin{align*}
\norm{u_2}_{L_{x_1}^2L_{\bar x_1,t}^{\infty}}\les \norm{u_2}_{\dot
X^{0,1/2,1}}\les \e^{-1/2}\norm{\widehat F}_{L^2}\les
\norm{F}_{L_{\ve e_1}^{1,2}}.
\end{align*}
Now we estimate $u_3$. Since $|\tau+|\xi|^2|\gg \e$, we have
\begin{align*}
\frac{1}{\tau+|\xi|^2+i\e
|\xi|^2}=\frac{1}{\tau+|\xi|^2}+\sum_{k=1}^\infty \frac{(-i\e
|\xi|^2)^k}{(\tau+|\xi|^2)^{k+1}}.
\end{align*}
Moreover, let $s=(-\tau-|\bar \xi|^2)^{1/2}$. Then
$\tau+|\xi|^2=-(s-\xi_1)(s+\xi_1)$, and thus we get $|s-\xi_1|\gg \e,
|s+\xi_1|\sim 1$
\[(\tau+|\xi|^2)^{-1}=-\frac{1}{2s}(\frac{1}{s-\xi_1}+\frac{1}{s+\xi_1})=-\frac{1}{2s(s-\xi_1)}(1+\frac{s-\xi_1}{s+\xi_1}).\]
Hence
\begin{align*}
\frac{1}{\tau+|\xi|^2+i\e
|\xi|^2}=&\frac{1}{\tau+|\xi|^2}+\sum_{k=1}^\infty \frac{(-i\e
|\xi|^2)^k}{(2s(\xi_1-s))^{k+1}}\\
&+\sum_{k=1}^\infty \frac{(-i\e
|\xi|^2)^k}{(2s(\xi_1-s))^{k+1}}[(1+\frac{s-\xi_1}{s+\xi_1})^{k+1}-1]\\
:=&a_1(\xi,\tau)+a_2(\xi,\tau)+a_3(\xi,\tau).
\end{align*}
Inserting this identity into the expression of $u_3$, then we have
$u_3=u_3^1+u_3^2+u_3^3$, where
\[u_3^j=\int_{\R^{n+1}}{e^{it\tau}e^{ix\xi}}a_j(\xi,\tau)\widehat{F}(\xi,\tau)1_{-\tau-|\bar\xi|^2\sim 1, |\tau+|\xi|^2|\gg \e}d\xi d\tau, \quad j=1,2,3.\]
For $u_3^1$, this corresponds to the case $\e=0$ which is proved in
\cite{IK2}. For $u_3^3$, we can control it similarly as $u_1$, since
\[|a_3(\xi,\tau)|\les \sum_{k=1}^\infty \frac{\e^k
|\xi|^{2k}}{(2|s(s-\xi_1)|)^{k+1}}(k+1)\frac{|s-\xi_1|}{|s+\xi_1|}\les
1.\]

It remains to control $u_3^2$. Let $G_k(x_1,\bar
\xi,\tau)=\ft_{x_1}^{-1}1_{-\tau-|\bar \xi|^2\sim 1}1_{|\xi|\sim
1}|\xi|^{2k}\widehat{F}$. Note that \[\norm{G_k}_{L_{\ve
e_1}^{1,2}}\les c^k\norm{F}_{L_{\ve e_1}^{1,2}}.\] Then
\begin{align*}
u_3^2=&\sum_{k=1}^\infty (-i\e)^k\int_\R \int_{\R^{n+1}}
\frac{e^{it\tau}e^{ix\xi}1_{|s-\xi_1|\gg \e}}{(2s(\xi_1-s))^{k+1}}[e^{-iy_1\xi_1}G_k(y_1,\bar \xi,\tau)]d\xi d\tau dy_1\\
=&\sum_{k=1}^\infty (-i\e)^k\int_\R T^k_{y_1}(G(y_1,\cdot))(t,x)
dy_1
\end{align*}
where \[T^k_{y_1}(f)(t,x)=\int_{\R^{n+1}}
\frac{e^{it\tau}e^{ix\xi}1_{|s-\xi_1|\gg
\e}}{(2s(\xi_1-s))^{k+1}}[e^{-iy_1\xi_1}f(y_1,\bar \xi,\tau)]d\xi
d\tau.\] We have
\begin{align*}
T^k_{y_1}(f)(t,x)=&\int_{\R^{n}} \int_\R
\frac{e^{i(x_1-y_1)\xi_1}1_{|s-\xi_1|\gg \e}}{(\xi_1-s)^{k+1}}
d\xi_1\cdot (2s)^{-k-1}{e^{it\tau}e^{i\bar x\cdot
\bar\xi}}[f(y_1,\bar \xi,\tau)]d\bar\xi d\tau\\
=&\int_\R \frac{e^{i(x_1-y_1)\xi_1}1_{|\xi_1|\gg \e}}{(\xi_1)^{k+1}}
d\xi_1\cdot \int_{\R^{n}} e^{i(x_1-y_1)s}
(2s)^{-k-1}{e^{it\tau}e^{i\bar x\cdot \bar\xi}}[f(y_1,\bar
\xi,\tau)]d\bar\xi d\tau.
\end{align*}
Then we get
\[|T^k_{y_1}(f)(t,x)|\les M^{-k}\e^{-k}\bigg|\int_{\R^{n}} e^{i(x_1-y_1)s}\cdot
(2s)^{-k-1}{e^{it\tau}e^{i\bar x\cdot \bar\xi}}[f(y_1,\bar
\xi,\tau)]d\bar\xi d\tau\bigg|.\] Making a change of variable
$\eta_1=s=\sqrt{-\tau-|\bar\xi|^2}$, $d\tau=-2\eta_1 d\eta_1$, we
obtain
\begin{align*}
&\int_{\R^{n}} e^{i(x_1-y_1)s}\cdot (2s)^{-k-1}{e^{it\tau}e^{i\bar
x\cdot \bar\xi}}[f(y_1,\bar \xi,\tau)]d\bar\xi d\tau\\
=&\int_{\R^{n}} e^{i(x_1-y_1)\eta_1}\cdot
(2\eta_1)^{-k}{e^{it(\eta_1^2+|\bar \xi|^2)}e^{i\bar x\cdot
\bar\xi}}[f(y_1,\bar \xi,\eta_1^2+|\bar \xi|^2)]d\bar\xi d\tau.
\end{align*}
Thus, by the linear estimate (see Lemma \ref{lem:linearStr}) we get
\begin{align*}
\norm{T^k_{y_1}(f)}_{_{L_{x_1}^2L_{\bar x,t}^{\infty}}}\les
M^{-k}\e^{-k} \norm{f}_2,
\end{align*}
which suffices to give the estimate for $u_3^2$. We complete the
proof of the lemma.
\end{proof}

\begin{lem} Let $n\geq 3$, $k\in \Z,\e\geq 0$. Assume $u,F$ solves the equation
\[u_t-i\Delta u-\e \Delta u=F(x,t), \quad u(x,0)=u_0.\]
Then for any $\ve e\in \cir^{n-1}$ we have
\begin{align}
\norm{P_ku}_{L_t^\infty L_x^2\cap L_t^2L_x^{\frac{2n}{n-2}}}\les
\norm{u_0}_{L^2}+\norm{F}_{N_k},
\end{align}
where the implicit constant is independent of $\e$.
\end{lem}
\begin{proof}
Since we have for $t>0$
\begin{align*}
\norm{e^{it\Delta+\e t\Delta}u_0}_{L_x^\infty}\les&
t^{-n/2}\norm{u_0}_{L_x^1}\\
\norm{e^{it\Delta+\e t\Delta}u_0}_{L_x^2}\les& \norm{u_0}_{L_x^2}
\end{align*}
where the implicit constant is independent of $\e$, then by the
abstract framework of Keel-Tao \cite{KT} we get the Strichartz
estimates
\begin{align*}
\norm{P_ku}_{L_t^\infty L_x^2\cap L_t^2L_x^{\frac{2n}{n-2}}}\les
\norm{u_0}_{L^2}+\norm{F}_{L_t^1L_x^2},
\end{align*}
with the implicit constant independent of $\e$.

By the same argument as in Step 2 of the proof of Lemma
\ref{lem:max}, we get
\begin{align*}
\norm{P_ku}_{L_t^\infty L_x^2\cap L_t^2L_x^{\frac{2n}{n-2}}}\les
\norm{u_0}_{L^2}+2^{-k/2}\sup_{\ve e\in \cir^{n-1}}\norm{F}_{L_{\ve
e}^{1,2}}.
\end{align*}
On the other hand, by Lemma \ref{lem:ext}, Lemma \ref{lem:linearStr} and Lemma \ref{lem:linearXsb}, we get
\begin{align*}
\norm{P_ku}_{L_t^\infty L_x^2\cap L_t^2L_x^{\frac{2n}{n-2}}}\les
\norm{P_ku}_{X_+^{0,1/2,1}}\les
\norm{u_0}_{L^2}+\norm{F}_{X_+^{0,-1/2,1}}.
\end{align*}
Thus we complete the proof.
\end{proof}

Gathering the above lemmas, we can get the following linear estimates:

\begin{lem}[Linear estimates]\label{lem:linear}
Assume $n\geq 3$, $u,F$ solves the equation: for $\e>0$
\[u_t-i\Delta u-\e \Delta u=F(x,t), \quad u(x,0)=u_0.\]
Then for $s\in \R$
\begin{align}
\norm{u}_{F^s\cap Z^s}\les&
\norm{u_0}_{\dot B^s_{2,1}}+\norm{F}_{N^s},
\end{align}
where the implicit constant is independent of $\e$.
\end{lem}

\section{Nonlinear estimates}

In this section we prove some nonlinear estimates.
The nonlinear term in the Landau-Lifshitz equation is
\[G(u)=\frac{\bar u}{1+|u|^2}\sum_{j=1}^n(\p_{x_j}u)^2.\]
By Taylor's expansion, if $\norm{u}_\infty <1$ we have
\[G(u)=\sum_{k=0}^\infty \bar u(-1)^k|u|^{2k}\sum_{j=1}^n(\p_{x_j}u)^2.\]
So we will need to do multilinear estimates.

\begin{lem}
(1) If $j\geq 2k-100$ and $X$ is a space-time translation invariant
Banach space, then $Q_{\leq j}P_k$ is bounded on $X$ with bound
independent of $j,k$.

(2) For any $j,k$, $Q_{\leq j}P_{k,\ve e}$ is bounded on $L_{\ve
e}^{p, 2}$ and $Q_{\leq j}$ is bounded on $L_t^pL_x^2$ for $1\leq p
\leq \infty$, with bound independent of $j,k$.
\end{lem}

\begin{proof}
See the proof of Lemma 5.4 in \cite{Guo}.
\end{proof}

\begin{lem}\label{lem:L2est}
Assume $n\geq 3$, $k_1,k_2,k_3\in \Z$. Then
\begin{align}
\norm{P_{k_1}uP_{k_2}v}_{L_{t,x}^{2}}\les&
2^{(n-1)k_1/2}2^{-k_2/2}\norm{P_{k_1}u}_{Y_{k_1}+F_{k_1}}\norm{P_{k_2}v}_{F_{k_2}},\\
\norm{P_{k_3}(P_{k_1}uP_{k_2}v)}_{L_{t,x}^{2}}\les&
2^{\frac{(n-2)\min(k_1,k_2,k_3)}{2}}\norm{P_{k_1}u}_{Y_{k_1}}\norm{P_{k_2}v}_{Y_{k_2}}.
\end{align}
\end{lem}
\begin{proof}
For the first inequality, we have
\begin{align*}
\norm{P_{k_1}uP_{k_2} v}_{L_{t,x}^{2}}
\les& \sum_{j=1}^n\norm{P_{k_1} uP_{k_2,\ve e_j}\Theta_{k_2}^jv}_{L_{t,x}^{2}}\les \sum_{j=1}^n\norm{P_{k_1} u}_{L_{\ve
e_j}^{2,\infty}}\norm{P_{k_2,\ve e_j}v}_{L_{\ve e_j}^{\infty,2}}\\
\les&2^{(n-1)k_1/2}2^{-k_2/2}\norm{P_{k_1}u}_{Y_{k_1}+F_{k_1}}\norm{P_{k_2}v}_{F_{k_2}}.
\end{align*}
For the second inequality, if $k_3\leq \min(k_1,k_2)$, then
\begin{align*}
\norm{P_{k_3}(P_{k_1}uP_{k_2}v)}_{L_{t,x}^{2}}\les&2^{k_3(n-2)/2}\norm{P_{k_1}uP_{k_2}v}_{L_{t}^{2}L_x^{\frac{2n}{2n-2}}}\\
\les&2^{k_3(n-2)/2}\norm{P_{k_1}u}_{L_{t}^{2}L_x^{\frac{2n}{n-2}}}\norm{P_{k_2}v}_{L_{t}^\infty L_x^2}.
\end{align*}
If $k_1\leq \min(k_2,k_3)$, then
\begin{align*}
\norm{P_{k_3}(P_{k_1}uP_{k_2}v)}_{L_{t,x}^{2}}
\les&\norm{P_{k_1}u}_{L_{t}^{2}L_x^\infty}\norm{P_{k_2}v}_{L_{t}^\infty L_x^2}\\
\les&2^{k_1(n-2)/2}\norm{P_{k_1}u}_{L_{t}^{2}L_x^{\frac{2n}{n-2}}}\norm{P_{k_2}v}_{L_{t}^\infty L_x^2}.
\end{align*}
If $k_2\leq \min(k_1,k_3)$, the proof is identical to the above case.
\end{proof}

\begin{lem}[Algebra properties]\label{lem:alg}
If $s\geq n/2$, then we have
\begin{align*}
\norm{uv}_{Y^{s}}\les&
\norm{u}_{Y^{s}}\norm{v}_{Y^{n/2}}+\norm{u}_{Y^{n/2}}\norm{v}_{Y^{s}}.
\end{align*}
\end{lem}
\begin{proof}
We only show the case $s=n/2$. By the embedding $\dot
B^{n/2}_{2,1}\subset L_x^\infty$ we get
\[\norm{u}_{L^\infty_{x,t}}\leq \norm{u}_{L^\infty_{t}\dot
B_{2,1}^{n/2}}\les \norm{u}_{Y^{n/2}}.\] The Lebesgue component can
be easily handled by para-product decomposition and H\"older's
inequality. Now we deal with $X^{s,b}$-type space.  By \eqref{eq:X01ext} it suffices to
show
\begin{align}
\sum_k2^{nk/2}2^{-k}\norm{P_k(fg)}_{X^{0,1}+\bar X^{0,1}}\les&
\norm{f}_{Y^{n/2}}\norm{g}_{Y^{n/2}},\label{eq:algeXsb}
\end{align}
We have
\begin{align*}
&\sum_k2^{nk/2}2^{-k}\norm{P_k(fg)}_{X^{0,1}+\bar X^{0,1}}\\
\les& \sum_{k_i}2^{nk_3/2}2^{-k_3}\norm{P_{k_3}(P_{k_1}fP_{k_2}g)}_{X^{0,1}+\bar X^{0,1}}\\
\les& (\sum_{k_i:k_1\leq k_2}+\sum_{k_i:k_1>
k_2})2^{k_3n/2}2^{-k_3}\norm{P_{k_3}(P_{k_1}fP_{k_2}g)}_{X^{0,1}+\bar X^{0,1}}\\
:=&I+II.
\end{align*}
By symmetry, we only need to estimate the term $I$.

Assume $P_{k_1}f=P_{k_1}f_1+P_{k_1}f_2$,
$P_{k_2}g=P_{k_2}g_1+P_{k_2}g_2$ such that
\begin{align*}
\norm{P_{k_1}f_1}_{X^{0,1}}+\norm{P_{k_1}f_2}_{\bar X^{0,1}}\les& \norm{P_{k_1}f}_{X^{0,1}+\bar X^{0,1}},\\
\norm{P_{k_2}g_1}_{X^{0,1}}+\norm{P_{k_2}g_2}_{\bar X^{0,1}}\les& \norm{P_{k_2}g}_{X^{0,1}+\bar X^{0,1}}.
\end{align*}
Then we have
\begin{align*}
I\les& \sum_{k_i:k_1\leq
k_2}\sum_{j=1}^22^{nk_3/2}2^{-k_3}\norm{P_{k_3}(P_{k_1}f_jP_{k_2}g_1)}_{X^{0,1}}\\
&+\sum_{k_i:k_1\leq
k_2}\sum_{j=1}^22^{nk_3/2}2^{-k_3}\norm{P_{k_3}(P_{k_1}f_jP_{k_2}g_2)}_{\bar X^{0,1}}\\
:=&I_1+I_2.
\end{align*}
We only estimate the term $I_1$ since the term $I_2$ can be
estimated in a similar way. We have
\begin{align}\label{eq:I1}
I_1\les&\sum_{k_i:k_1\leq
k_2}\sum_{j=1}^22^{nk_3/2}2^{-k_3}\norm{P_{k_3}(P_{k_1}f_jP_{k_2}g_1)}_{X^{0,1}}.
\end{align}
First we assume $k_3\leq k_1+5$ in the summation of \eqref{eq:I1}. We
have
\begin{align*}
I_1\les& \sum_{k_i:k_1\leq k_2}\sum_{j=1}^22^{nk_3/2}2^{-k_3}(\norm{P_{k_3}Q_{\leq k_1+k_2+9}(P_{k_1}f_jP_{k_2}g_1)}_{X^{0,1}}\\
&\qquad+\norm{P_{k_3}Q_{\geq k_1+k_2+10}(P_{k_1}f_jP_{k_2}g_1)}_{X^{0,1}})\\
:=&I_{11}+I_{12}.
\end{align*}
For the term $I_{11}$ we have
\begin{align*}
I_{11}\les &\sum_{k_i:k_1\leq
k_2}\sum_{j=1}^22^{k_3n/2}2^{-k_3}2^{k_1+k_2}\norm{P_{k_1}f_j}_{L_t^\infty
L_{x}^n}\norm{P_{k_2}g_1}_{L_t^2L_x^{\frac{2n}{n-2}}}\\
\les& \sum_{k_i:k_1\leq
k_2}\sum_{j=1}^22^{k_3n/2}2^{-k_3}2^{k_2}2^{k_1n/2}\norm{P_{k_1}f_j}_{L_t^\infty
L_{x}^2}\norm{P_{k_2}g_1}_{L_t^2L_x^{\frac{2n}{n-2}}}\\
\les&\norm{f}_{Y^{n/2}}\norm{g}_{Y^{n/2}}.
\end{align*}
For the term $I_{12}$, we need to exploit the nonlinear interactions as in \cite{Guo}. We have
\begin{align*}
&\ft P_{k_3}Q_{\geq k_1+k_2+10}(P_{k_1}f_jP_{k_2}g_1)\\
&=\chi_{k_3}(\xi_3)\chi_{\geq k_1+k_2+10}(\tau_3+|\xi_3|^2)
\int_{\xi_3=\xi_1+\xi_2,\tau_3=\tau_1+\tau_2}\chi_{k_1}(\xi_1)\widehat{f_j}(\tau_1,\xi_1)\chi_{k_2}(\xi_2)\widehat{g_1}(\tau_2,\xi_2).
\end{align*}
We assume $j=1$ since $j=2$ is similar.  On the plane $\{\xi_3=\xi_1+\xi_2,\tau_3=\tau_1+\tau_2\}$ we have
\begin{align}
\tau_3+|\xi_3|^2=\tau_1+|\xi_1|^2+\tau_2+|\xi_2|^2-H(\xi_1,\xi_2)
\end{align}
where $H$ is the resonance function in the product
$P_{k_3}(P_{k_1}f_jP_{k_2}g_1)$
\begin{align}
H(\xi_1,\xi_2)=|\xi_1|^2+|\xi_2|^2-|\xi_1+\xi_2|^2.
\end{align}
Since $|H|\les 2^{k_1+k_2}$, then
one of $P_{k_1}f_j$, $P_{k_2}g_1$ has modulation larger than the
output modulation, namely
\[\max(|\tau_1+|\xi_1|^2|,|\tau_2+|\xi_2|^2|)\ges |\tau_3+|\xi_3|^2|.\]
If $P_{k_1}f_j$ has larger modulation, then
\begin{align*}
I_{12}\les& \sum_{k_i:k_1\leq k_2}2^{nk_3/2}2^{-k_3}\norm{2^{j_3}\norm{P_{k_3}Q_{j_3}(P_{k_1}f_jP_{k_2}g_1)}_{L^2_{t,x}}}_{l^2_{j_3\geq k_1+k_2+10}}\\
\les&
\sum_{k_i:k_1\leq k_2}\sum_{j=1}^22^{nk_3/2}2^{nk_3/2}2^{-k_3}(\sum_{j_3\geq k_1+k_2+10}2^{2j_3}\norm{Q_{\geq j_3}P_{k_1}f_j}^2_{L_{t,x}^2})^{1/2}\norm{P_{k_2}g_1}_{L^\infty_tL_x^2}\\
\les&\sum_{k_i:k_1\leq
k_2}2^{nk_3}2^{-k_3}(\norm{P_{k_1}f_1}_{X^{0,1}}+\norm{P_{k_1}f_2}_{\bar X^{0,1}})\norm{P_{k_2}g_1}_{Y_{k_2}}\\
\les&
\norm{f}_{Y^{n/2}}\norm{g}_{Y^{n/2}}.
\end{align*}
If $P_{k_2}g_1$ has larger modulation, then
\begin{align*}
I_{12}\les& \sum_{k_i:k_1\leq k_2}\sum_{j=1}^22^{nk_3/2}2^{-k_3}\norm{P_{k_1}f_j}_{L_{t,x}^\infty}(\sum_{j_3\geq k_1+k_2}2^{2j_3}\norm{P_{k_2}Q_{\geq j_3}g_1}^2_{L^2_tL_x^2})^{1/2}\\
\les&\sum_{k_i:k_1\leq
k_2}\sum_{j=1}^22^{nk_3/2}2^{-k_3}2^{nk_1/2}\norm{P_{k_1}f_j}_{Y_{k_1}}\norm{P_{k_2}g_1}_{X^{0,1}}\\
\les&
\norm{f}_{Y^{n/2}}\norm{g}_{Y^{n/2}}.
\end{align*}

Now we assume $k_3\geq k_1+6$ in the summation of \eqref{eq:I1} and thus $|k_2-k_3|\leq 4$.  We have
\begin{align*}
I_1\les& \sum_{k_i:k_1\leq k_2}\sum_{j=1}^22^{nk_3/2}2^{-k_3}(\norm{P_{k_3}Q_{\leq k_1+k_2+9}(P_{k_1}f_jP_{k_2}g_1)}_{X^{0,1}}\\
&\qquad+\norm{P_{k_3}Q_{\geq k_1+k_2+10}(P_{k_1}f_jP_{k_2}g_1)}_{X^{0,1}})\\
:=&\tilde I_{11}+\tilde I_{12}.
\end{align*}
By Lemma \ref{lem:L2est} we get
\begin{align*}
\tilde I_{11}\les& \sum_{k_i:k_1\leq k_2}\sum_{j=1}^22^{nk_3/2}2^{-k_3}(\norm{P_{k_3}Q_{\leq k_1+k_2+9}(P_{k_1}f_jP_{k_2}g_1)}_{X^{0,1}}\\
\les&\sum_{k_i:k_1\leq
k_2}2^{nk_3/2}2^{k_1}2^{(n-2)k_1/2}\norm{P_{k_1}f_j}_{Y_{k_1}}\norm{P_{k_2}g_1}_{Y_{k_2}}\les
\norm{f}_{Y^{n/2}}\norm{g}_{Y^{n/2}}.
\end{align*}
For the term $\tilde I_{12}$, similarly as the term $I_{12}$,
one of $P_{k_1}f_j$, $P_{k_2}g_1$ has modulation larger than the
output modulation. If $P_{k_1}f_j$ has larger modulation, then
\begin{align*}
\tilde I_{12}\les& \sum_{k_i:k_1\leq
k_2}\sum_{j=1}^22^{nk_3/2}2^{-k_3}(\sum_{j_3}2^{2j_3}\norm{P_{k_1}Q_{\geq
j_3}f_j}^2_{L_t^2L_x^\infty})^{1/2} \norm{P_{k_2}g_1}_{L_t^\infty
L_x^2}\\
\les&\norm{f}_{Y^{n/2}}\norm{g}_{Y^{n/2}}.
\end{align*}
If $P_{k_2}g_1$ has larger modulation, then
\begin{align*}
\tilde I_{12}\les& \sum_{k_i:k_1\leq
k_2}\sum_{j=1}^22^{nk_3/2}2^{-k_3}\norm{P_{k_1}f_j}_{L_t^\infty L_x^\infty}
(\sum_{j_3}2^{2j_3}\norm{P_{k_2}Q_{\geq j_3}g_1}^2_{L_t^2
L_x^2})^{1/2}\\
\les& \norm{f}_{Y^{n/2}}\norm{g}_{Y^{n/2}}.
\end{align*}
Thus, we complete the proof.
\end{proof}

\begin{lem}\label{lem:L2non}
We have
\begin{align}\label{eq:nonL2est}
\sum_{k_j}2^{k_3(n-2)/2}\norm{P_{k_3}[u\sum_{i=1}^n (\p_{x_i}P_{k_1}v
\p_{x_i}P_{k_2}w)]}_{L_{t,x}^2}\les
\norm{u}_{Y^{n/2}}\norm{v}_{F^{n/2}}\norm{w}_{F^{n/2}}.
\end{align}
\end{lem}
\begin{proof}
We have
\begin{align*}
\mbox{LHS of \eqref{eq:nonL2est}}\les&
\sum_{k_j}2^{k_3(n-2)/2}\norm{P_{k_3}[P_{\geq
k_3-10}u\sum_{i=1}^n
(\p_{x_i}P_{k_1}v \p_{x_i}P_{k_2}w)]}_{L_{t,x}^2}\\
&+ \sum_{k_j}2^{k_3(n-2)/2}\norm{P_{k_3}[P_{\leq k_3-10}u\sum_{i=1}^n
(\p_{x_i}P_{k_1}v \p_{x_i}P_{k_2}w)]}_{L_{t,x}^2}\\
:=&I+II.
\end{align*}
By symmetry, we may assume $k_1\leq k_2$ in the above summation.
For the term $II$, since $n\geq 3$, then we have
\begin{align*}
II\les& \norm{u}_{Y^{n/2}} \sum_{k_j}2^{k_3(n-2)/2}\norm{\tilde
P_{k_3}\sum_{i=1}^n (\p_{x_i}P_{k_1}v
\p_{x_i}P_{k_2}w)]}_{L_{t,x}^2}\\
\les&\norm{u}_{Y^{n/2}} \sum_{k_1,k_3\leq k_2+5}2^{k_3(n-2)/2}2^{k_1+k_2}2^{[(n-1)k_1-k_2]/2}\norm{P_{k_1}v}_{F_{k_1}}\norm{P_{k_2}w}_{F_{k_2}}\\
\les& \norm{u}_{Y^{n/2}}\norm{v}_{F^{n/2}}\norm{w}_{F^{n/2}}.
\end{align*}
For the term $I$, if $k_3\leq k_2+20$, then we get from Lemma \ref{lem:L2est} that
\begin{align*}
I\les& \sum_{k_j}2^{nk_3/2}2^{k_3(n-2)/2}\norm{P_{\geq k_3-10}u}_{L_t^\infty
L_x^2}\norm{\sum_{i=1}^n (\p_{x_i}P_{k_1}v
\p_{x_i}P_{k_2}w)]}_{L_{t,x}^2}\\
\les&\sum_{k_j}2^{nk_3/2}2^{k_3(n-2)/2}2^{(n+1)k_1/2}2^{k_2/2}\norm{P_{\geq
k_3-10}u}_{L_t^\infty L_x^2}\norm{P_{k_1}v}_{F_{k_1}}\norm{P_{k_2}w}_{F_{k_2}}\\
\les& \norm{u}_{Y^{n/2}}\norm{v}_{F^{n/2}}\norm{w}_{F^{n/2}}.
\end{align*}
If $k_3\geq k_2+20$, then $u$ has frequency $\sim 2^{k_3}$, and thus we get
\begin{align*}
I\les& \sum_{k_j}2^{k_3(n-2)/2}\norm{P_{k_3}u}_{L_t^\infty
L_x^2}\norm{\sum_{i=1}^n (\p_{x_i}P_{k_1}v
\p_{x_i}P_{k_2}w)]}_{L_{t}^2L_x^\infty}\\
\les& \sum_{k_j}2^{k_3(n-2)/2}\norm{P_{k_3}u}_{L_t^\infty
L_x^2}2^{nk_2/2}\norm{\sum_{i=1}^n (\p_{x_i}P_{k_1}v
\p_{x_i}P_{k_2}w)]}_{L_{t}^2L_x^2}\\
\les&\sum_{k_j}2^{k_3(n-2)/2}2^{k_1+k_2}2^{(n-1)k_1/2}2^{(n-1)k_2/2}\norm{P_{k_3}u}_{L_t^\infty L_x^2}\norm{P_{k_1}v}_{F_{k_1}}\norm{P_{k_2}w}_{F_{k_2}}\\
\les& \norm{u}_{Y^{n/2}}\norm{v}_{F^{n/2}}\norm{w}_{F^{n/2}}.
\end{align*}
Therefore we complete the proof.
\end{proof}

\begin{lem}\label{lem:Nnon}
We have
\begin{align}
\norm{u\sum_{i=1}^n(\p_{x_i}v \partial_{x_i}w)}_{N^{n/2}}\les&
\norm{u}_{Y^{n/2}}\norm{v}_{F^{n/2}\cap Z^{n/2}}\norm{w}_{F^{n/2}\cap Z^{n/2}}.
\end{align}
\end{lem}
\begin{proof}
By the definition of $N^{n/2}$, the $L^2$ component was handled by the previous lemma. We only need to control
\begin{align}\label{eq:tripf1}
\sum_{k_i}2^{k_4n/2}\norm{P_{k_4}[P_{k_1}u\sum_{i=1}^n(P_{k_2}\p_{x_i}v
\partial_{x_i}P_{k_3}w)]}_{N_{k_4}}.
\end{align}
By symmetry we may assume $k_2\leq k_3$ in the above summation.
If in the above summation we assume $k_4\leq k_1+40$, then
\begin{align*}
\eqref{eq:tripf1}\les& \sum_{k_i}2^{k_4n/2}\norm{P_{k_4}[P_{k_1}uP_{k_2}\p_{x_i}v
\partial_{x_i}P_{k_3}w]}_{L_t^1L_x^{2}}\\
\les& \sum_{k_i}2^{k_1n/2}\norm{P_{k_1}u}_{L_t^\infty L_x^2}2^{k_2}\norm{P_{k_2}v}_{L_t^2L_x^{\infty}}2^{k_3}\norm{P_{k_3}w}_{L_t^2L_x^{\infty}}\\
\les&\sum_{k_i}2^{k_1n/2}\norm{P_{k_1}u}_{L_t^\infty L_x^2}2^{k_2n/2}\norm{P_{k_2}v}_{L_{t}^{2}L_x^{\frac{2n}{n-2}}}2^{k_3n/2}\norm{P_{k_3}w}_{L_{t}^{2}L_x^{\frac{2n}{n-2}}}\\
\les&\norm{u}_{Y^{n/2}}\norm{v}_{F^{n/2}}\norm{w}_{F^{n/2}}.
\end{align*}
Thus from now on we assume $k_4\geq k_1+40$ in the summation of \eqref{eq:tripf1}. We bound the summation case by case.

{\bf Case 1:} $k_2\leq k_1+20$

In this case we have $k_4\geq k_2+20$ and hence $|k_4-k_3|\leq 5$. By Lemma \ref{lem:L2est} we
get
\begin{align*}
\eqref{eq:tripf1}\les&
\sum_{k_i}2^{k_4n/2}2^{-k_4/2}\norm{P_{k_4}[P_{k_1}uP_{k_2}\p_{x_i}
v
\partial_{x_i}P_{k_3}w]}_{L_{\ve{e}}^{1,2}}\\
\les& \sum_{k_i}2^{k_4n/2}2^{-k_4/2}\norm{P_{k_1}uP_{k_3}\partial_{x_i}w}_{L_{x,t}^{2}}\norm{P_{k_2}\p_{x_i}v}_{L_{\ve{e}}^{2,\infty}}\\
\les& \sum_{k_i}2^{k_4n/2}2^{-k_4/2}2^{(n-1)k_1/2}2^{-k_3/2}2^{(n-1)k_2/2}\norm{P_{k_1}u}_{Y_{k_1}}\norm{P_{k_3}\partial_{x_i}w}_{F_{k_3}}\norm{P_{k_2}\p_{x_i}v}_{F_{k_2}}\\
\les&\norm{u}_{Y^{n/2}}\norm{v}_{F^{n/2}}\norm{w}_{F^{n/2}}.
\end{align*}

{\bf Case 2:} $k_2\geq k_1+21$

In this case we have $k_4\leq k_3+40$. Let $g=\sum_{i=1}^n(P_{k_2}\p_{x_i}v\cdot
P_{k_3}\partial_{x_i}w)$. Then we have
\begin{align*}
\eqref{eq:tripf1}\les&\sum_{k_i}2^{k_4n/2}\norm{P_{k_4}[P_{k_1}uQ_{\leq k_2+k_3}g]}_{N_{k_4}}+\sum_{k_i}2^{k_4n/2}\norm{P_{k_4}[P_{k_1}uQ_{\geq k_2+k_3}g]}_{N_{k_4}}\\
:=&I+II.
\end{align*}
First we estimate the term $II$.  We have
\begin{align*}
II\les &\sum_{k_i}2^{k_4n/2}\norm{P_{k_4}[P_{k_1}Q_{\geq k_2+k_3-10}u\cdot Q_{\geq k_2+k_3}g]}_{N_{k_4}}\\
&+\sum_{k_i}2^{k_4n/2}\norm{P_{k_4}[P_{k_1}Q_{\leq k_2+k_3-10}u\cdot Q_{\geq k_2+k_3}g]}_{N_{k_4}}\\
:=&II_1+II_2.
\end{align*}
For the term $II_1$ we have
\begin{align*}
II_1\les &\sum_{k_i}2^{k_4n/2}\norm{P_{k_4}[P_{k_1}Q_{\geq k_2+k_3-10}u\cdot Q_{\geq k_2+k_3}g]}_{L_t^1L_x^2}\\
\les &\sum_{k_i}2^{k_4n/2}\norm{P_{k_1}Q_{\geq k_2+k_3-10}u}_{L_t^2L_x^\infty}\norm{Q_{\geq k_2+k_3}g]}_{L_t^2L_x^2}\\
\les &\sum_{k_i}2^{k_4n/2}2^{k_1n/2}\norm{P_{k_1}Q_{\geq k_2+k_3-10}u}_{L_t^2L_x^2}\norm{Q_{\geq k_2+k_3}g]}_{L_t^2L_x^2}\\
\les &\sum_{k_i}2^{k_4n/2}2^{k_1n/2}2^{-(k_2+k_3)}\norm{P_{k_1}Q_{\geq 2k_1+10}u}_{X^{0,1}}\\
&\cdot 2^{[(n-1)k_2-k_3]/2}2^{k_2+k_3}\norm{P_{k_2}v}_{F_{k_2}}\norm{P_{k_3}w}_{F_{k_3}}\\
\les&\norm{u}_{Y^{n/2}}\norm{v}_{F^{n/2}}\norm{w}_{F^{n/2}}.
\end{align*}
For the term $II_2$, since $k_4\geq k_1+40$, then we may assume $g$
has frequency of size $2^{k_4}$.  The resonance function in the
product $P_{k_1}u\cdot P_{k_4}g$ is of size $\les 2^{k_1+k_4}$.
 Thus the output modulation is of size $\ges 2^{k_2+k_3}$.  Then we get
\begin{align*}
II_2\les &\sum_{k_i}2^{k_4n/2}2^{-(k_2+k_3)/2}\norm{P_{k_4}[P_{k_1}Q_{\leq k_2+k_3-10}u\cdot Q_{\geq k_2+k_3}g]}_{L^2_{t,x}}\\
\les &\sum_{k_i}2^{k_4n/2}2^{-(k_2+k_3)/2}2^{k_1n/2}\norm{P_{k_1}u}_{L_t^\infty L_x^2}\cdot \norm{g}_{L^2_{t,x}}\\
\les&\sum_{k_i}2^{k_4n/2}2^{-(k_2+k_3)/2}2^{k_1n/2}2^{[(n-1)k_2-k_3]/2}2^{k_2+k_3} \norm{P_{k_1}u}_{Y_{k_1}}\norm{P_{k_2}v}_{F_{k_2}}\norm{P_{k_3}w}_{F_{k_3}}\\
\les&\norm{u}_{Y^{n/2}}\norm{v}_{F^{n/2}}\norm{w}_{F^{n/2}}.
\end{align*}

Now we estimate the term $I$.  We have
\begin{align*}
I\les &\sum_{k_i}2^{k_4n/2}\norm{P_{k_4}[P_{k_1}u\cdot Q_{\leq
k_2+k_3}\sum_{i=1}^n(P_{k_2}\p_{x_i}Q_{\geq k_2+k_3+40}v\cdot
P_{k_3}\partial_{x_i}w)]}_{N_{k_4}}\\
&+\sum_{k_i}2^{k_4n/2}\norm{P_{k_4}[P_{k_1}u\cdot Q_{\leq
k_2+k_3}\sum_{i=1}^n(P_{k_2}\p_{x_i}Q_{\leq k_2+k_3+39}v\cdot
P_{k_3}\partial_{x_i}w)]}_{N_{k_4}}\\
:=&I_1+I_2.
\end{align*}
For the term $I_1$, since the resonance function in the product
$P_{k_2}v\cdot P_{k_3}w$ is of size $\les 2^{k_2+k_3}$, then we may
assume $P_{k_3}w$ has modulation of size $\ges 2^{k_2+k_3}$.  Then
we get
\begin{align*}
I_1\les& \sum_{k_i}2^{k_4n/2}2^{-k_4/2}\norm{P_{k_4}[P_{k_1}u\cdot
Q_{\leq k_2+k_3}\sum_{i=1}^n(P_{k_2}\p_{x_i}Q_{\geq
k_2+k_3+40}v\cdot
P_{k_3}\partial_{x_i}Q_{\geq k_2+k_3-5}w)]}_{L^{1,2}_{\ve e}}\\
\les&\sum_{k_i}2^{k_4n/2}2^{-k_4/2}\norm{P_{k_1}u}_{L^\infty_{t,x}}2^{k_2+k_3}\norm{P_{k_2}v}_{L^{2,\infty}_{\ve e}}\norm{P_{k_3}Q_{\geq k_2+k_3-5}w}_{L^2_{t,x}}\\
\les&\sum_{k_i}2^{k_4n/2}2^{-k_4/2}2^{(k_2+k_3)/2}2^{(n-1)k_2/2}2^{k_1n/2}\norm{P_{k_1}u}_{Y_{k_1}}\norm{P_{k_2}v}_{F_{k_2}}\norm{P_{k_3}w}_{F_{k_3}}\\
\les&\norm{u}_{Y^{n/2}}\norm{v}_{F^{n/2}}\norm{w}_{F^{n/2}}.
\end{align*}
Finally, we estimate the term $I_2$.  For this term, we need to use
the null structure observed by Bejenaru \cite{Bej}.  We can rewrite
\begin{align}
-2\nabla u\cdot \nabla v=(i\p_t+\Delta)u\cdot v+u\cdot
(i\p_t+\Delta)v-(i\p_t+\Delta)(u\cdot v).
\end{align}
Then we have
\begin{align*}
I_2=&\sum_{k_i}2^{k_4n/2}\norm{P_{k_4}[P_{k_1}u\cdot Q_{\leq
k_2+k_3}(P_{k_2}LQ_{\leq k_2+k_3+39}v\cdot
P_{k_3}w)]}_{N_{k_4}}\\
&+\sum_{k_i}2^{k_4n/2}\norm{P_{k_4}[P_{k_1}u\cdot Q_{\leq
k_2+k_3}(P_{k_2}Q_{\leq k_2+k_3+39}v\cdot
P_{k_3}Lw)]}_{N_{k_4}}\\
&+\sum_{k_i}2^{k_4n/2}\norm{P_{k_4}[P_{k_1}u\cdot Q_{\leq
k_2+k_3}L(P_{k_2}Q_{\leq k_2+k_3+39}v\cdot
P_{k_3}w)]}_{N_{k_4}}\\
:=&I_{21}+I_{22}+I_{23}.
\end{align*}
For the term $I_{21}$, we have
\begin{align*}
I_{21}\les& \sum_{k_i}2^{k_4n/2}\norm{P_{k_4}[P_{k_1}u\cdot Q_{\leq
k_2+k_3}(P_{k_2}L Q_{\leq k_2+k_3+39}v\cdot P_{k_3}w)]}_{L_t^{1}L^2_{x}}\\
\les&\sum_{k_i}2^{k_4n/2}2^{k_1n/2}\norm{P_{k_1}u}_{L_t^\infty
L_x^2}2^{k_2(n-2)/2}\norm{P_{k_2}L v}_{L^2_{t,x}}\norm{P_{k_3}w}_{L_t^{2}L_x^{\frac{2n}{n-2}}}\\
\les&\norm{u}_{Y^{n/2}}\norm{v}_{Z^{n/2}}\norm{w}_{F^{n/2}}.
\end{align*}
For the term $I_{22}$, we may assume  $w$ has modulation $\les 2^{k_2+k_3}$. Then we get
\begin{align*}
I_{22}\les&\sum_{k_i}2^{k_4n/2}2^{-k_4/2}\norm{P_{k_4}[P_{k_1}u\cdot
Q_{\leq k_2+k_3}(P_{k_2}Q_{\leq k_2+k_3+39}v\cdot
P_{k_3}Q_{\leq k_2+k_3+100}Lw)]}_{L^{1,2}_{\ve e}}\\
\les&\sum_{k_i}2^{k_4n/2}2^{-k_4/2}2^{k_1n/2}\norm{P_{k_1}u}_{L_t^\infty
L_x^2}\norm{P_{k_2}v}_{L^{2,\infty}_{\ve e}}\norm{
P_{k_3}Q_{\leq k_2+k_3+100}Lw)]}_{L^{2}_{t,x}}\\
\les&\sum_{k_i}2^{k_4n/2}2^{-k_4/2}2^{nk_1/2}2^{(n-1)k_2/2}2^{(k_2+k_3)/2}\norm{P_{k_1}u}_{Y_{k_1}}\norm{P_{k_2}v}_{F_{k_2}}\norm{P_{k_3}w}_{X^{0,1/2,\infty}}\\
\les&\norm{u}_{Y^{n/2}}\norm{v}_{F^{n/2}}\norm{w}_{F^{n/2}}.
\end{align*}
Next we estimate the term $I_{23}$. We have
\begin{align*}
I_{23}\les& \sum_{k_i}2^{k_4n/2}\norm{P_{k_4}[P_{k_1}u\cdot
Q_{[k_1+k_4+100, k_2+k_3]}L(P_{k_2}Q_{\leq k_2+k_3+39}v\cdot
P_{k_3}w)]}_{N_{k_4}} \\
&+ \sum_{k_i}2^{k_4n/2}\norm{P_{k_4}[P_{k_1}u\cdot Q_{\leq
k_1+k_4+99}L(P_{k_2}Q_{\leq k_2+k_3+39}v\cdot
P_{k_3}w)]}_{N_{k_4}} \\
:=&I_{231}+I_{232}.
\end{align*}
For the term $I_{232}$ we have
\begin{align*}
I_{232}\les& \sum_{k_i}2^{k_4n/2}2^{-k_4/2}\norm{P_{k_4}[P_{k_1}u\cdot
Q_{\leq k_1+k_4+99}L(P_{k_2}Q_{\leq k_2+k_3+39}v\cdot
P_{k_3}w)]}_{L^{1,2}_{\ve e}} \\
\les & \sum_{k_i}2^{k_4n/2}2^{-k_4/2}\norm{P_{k_1}u}_{L^{2,\infty}_{\ve e}}2^{k_1+k_4}2^{[(n-1)k_2-k_3]/2}\norm{P_{k_2}v}_{F_{k_2}}\norm{P_{k_3}w}_{F_{k_3}}\\
\les&\norm{u}_{Y^{n/2}}\norm{v}_{F^{n/2}}\norm{w}_{F^{n/2}}.
\end{align*}
For the term $I_{231}$ we have
\begin{align*}
I_{231}\les&
\sum_{k_i}\sum_{j_2=k_1+k_4+100}^{k_2+k_3}2^{k_4n/2}\norm{P_{k_4}Q_{\leq
j_2-10}[P_{k_1}u\cdot Q_{j_2}L(P_{k_2}Q_{\leq k_2+k_3+39}v\cdot
P_{k_3}w)]}_{N_{k_4}}\\
&+\sum_{k_i}\sum_{j_2=k_1+k_4+100}^{k_2+k_3}2^{k_4n/2}\norm{P_{k_4}Q_{\geq
j_2-9}[P_{k_1}u\cdot Q_{j_2}L(P_{k_2}Q_{\leq k_2+k_3+39}v\cdot
P_{k_3}w)]}_{N_{k_4}}\\
:=&I_{2311}+I_{2312}.
\end{align*}
For the term $I_{2312}$ we have
\begin{align*}
I_{2312}\les&
\sum_{k_i}\sum_{j_2=k_1+k_4+100}^{k_2+k_3}\sum_{j_3\geq
k_2-9}2^{k_4n/2}2^{-j_3/2}\\
&\cdot\norm{P_{k_4}Q_{j_3}[P_{k_1}u\cdot Q_{j_2}L(P_{k_2}Q_{\leq
k_2+k_3+39}v\cdot
P_{k_3}w)]}_{L^2_{t,x}}\\
\les& \sum_{k_i}2^{k_4n/2}2^{k_1n/2}2^{(k_2+k_3)/2}2^{[(n-1)k_2-k_3]/2}\norm{P_{k_1}u}_{Y_{k_1}}\norm{P_{k_2}v}_{F_{k_2}}\norm{P_{k_3}w}_{F_{k_3}}\\
\les&\norm{u}_{Y^{n/2}}\norm{v}_{F^{n/2}}\norm{w}_{F^{n/2}}.
\end{align*}
For the term $I_{2311}$ we have
\begin{align*}
I_{2311}\les&
\sum_{k_i}\sum_{j_2=k_1+k_4+100}^{k_2+k_3}2^{k_4n/2}\norm{P_{k_4}Q_{\leq
j_2-10}[P_{k_1}\tilde Q_{j_2}u\cdot Q_{j_2}L(P_{k_2}Q_{\leq
k_2+k_3+39}v\cdot
P_{k_3}w)]}_{L_t^1L_x^2}\\
\les&\sum_{k_i}\sum_{j_2=k_1+k_4+100}^{k_2+k_3}2^{k_4n/2}2^{k_1n/2}\norm{P_{k_1}\tilde Q_{j_2}u}_{L^2_{t,x}}2^{j_2}2^{[(n-1)k_2-k_3]/2}\norm{P_{k_2}v}_{F_{k_2}}\norm{P_{k_3}w}_{F_{k_3}}\\
\les&\norm{u}_{Y^{n/2}}\norm{v}_{F^{n/2}}\norm{w}_{F^{n/2}}.
\end{align*}
Therefore, we complete the proof.
\end{proof}

Combining all the estimates above we get
\begin{lem}[Nonlinear estimates]\label{lem:nonlinear}
Assume $u\in F^{n/2}\cap Z^{n/2}$ with $\norm{u}_{Y^{n/2}}\ll 1$. Then
\begin{align*}
\normo{\frac{\bar
u}{1+|u|^2}\sum_{j=1}^n(\p_{x_j}u)^2}_{N^{n/2}}\les&
\frac{\norm{u}_{Y^{n/2}}}{1-\norm{u}_{Y^{n/2}}^2}
\norm{u}_{F^{n/2}\cap Z^{n/2}}\norm{u}_{F^{n/2}\cap Z^{n/2}}.
\end{align*}
\end{lem}
\begin{proof}
Since $Y^{n/2}\subset L^\infty$, then
\[\frac{\bar
u}{1+|u|^2}\sum_{j=1}^n(\p_{x_j}u)^2=\sum_{k=0}^\infty \bar u(-1)^k|u|^{2k}\sum_{j=1}^n(\p_{x_j}u)^2.\]
The lemma follows from Lemma \ref{lem:Nnon}, Lemma \ref{lem:L2non} and Lemma \ref{lem:alg}.
\end{proof}

\section{The limit behaviour}

In this section we prove Theorem \ref{thm:Schmap}. It is equivalent
to prove

\begin{thm}\label{thm:dGL}
Assume $n\geq 3$, $\e\in [0,1]$. There exists $0<\delta\ll 1$ such
for any $\phi\in \dot B_{2,1}^{n/2}$ with $\norm{\phi}_{\dot
B_{2,1}^{n/2}}\leq \delta$, there exists a unique global solution
$u_\e$ to \eqref{eq:dGL} such that
\[\norm{u_\e}_{F^{n/2}\cap Z^{n/2}}\les
\delta,\] where the implicit constant is independent of $\e$. The
map $\phi\to u_\e$ is Lipshitz from $B_\delta(\dot B_{2,1}^{n/2})$
to $C(\R;\dot B_{2,1}^{n/2})$ and the Lipshitz constant is
independent of $\e$. Moreover, for any $T>0$,
\[\lim_{\e\to 0+}\norm{u_\e-u_0}_{C([0,T];\dot B_{2,1}^{n/2})}=0.\]
\end{thm}

For the uniform global well-posedness, we can prove it by standard
Picard iteration argument by using the linear and nonlinear
estimates proved in the previous section.
Indeed, define
\begin{align*}
\Phi_{u_0}(u):=&e^{it\Delta+\e t\Delta}u_0\\
&-i\int_0^t e^{i(t-s)\Delta+\e (t-s)\Delta}\bigg[\frac{2\bar
u}{1+|u|^2}\sum_{j=1}^n(\p_{x_j}u)^2-\frac{2i\e\bar
u}{1+|u|^2}\sum_{j=1}^n(\p_{x_j}u)^2\bigg]ds.
\end{align*}
Then using the Lemma \ref{lem:linear} and Lemma \ref{lem:nonlinear}
we can show $\Phi_{u_0}$ is a contraction mapping in the set
\[\{u:\norm{u}_{F^{n/2}\cap Z^{n/2}}\leq C\delta\}\]
if $\norm{u_0}_{\dot B^{n/2}_{2,1}}\leq \delta$ with $\delta>0$ sufficiently small.  Thus we have existence and uniqueness.
Moreover, by standard arguments we immediately have the persistence of regularity, namely if $u_0\in \dot B^{s}_{2,1}$ for some $s>n/2$, then $u\in F^s\cap Z^s$ and
\begin{align}\label{eq:pers}
\norm{u}_{F^s\cap Z^s}\les \norm{u_0}_{\dot B^{s}_{2,1}}
\end{align}
uniformly with respect to $\e \in (0,1]$.

Now we prove the limit
behaviour.  Assume $u_\e$ is a solution to the Landau-Lifshitz
equation with small initial data $\phi_1\in \dot B^{n/2}_{2,1}$, and
$u$ is a solution to the Schr\"odinger map with small initial data
$\phi_2\in \dot B^{n/2}_{2,1}$. Let $w=u_\e-u$,
$\phi=\phi_1-\phi_2$, then $w$ solves
\begin{align}\label{eq:dGLw}
(i\p_t+\Delta)w=&i\e\Delta u_\e+\bigg[\frac{2\bar u_\e}{1+|u_\e|^2}\sum_{j=1}^n(\p_{x_j}u_\e)^2-\frac{2\bar u}{1+|u|^2}\sum_{j=1}^n(\p_{x_j}u)^2\bigg]\nonumber\\
&-\frac{2i\e\bar u_\e}{1+|u_\e|^2}\sum_{j=1}^n(\p_{x_j}u_\e)^2,\\
w(0)=&\phi.\nonumber
\end{align}
First we assume in addition $\phi_1 \in \dot B^{(n+4)/2}_{2,1}$. By
the linear and nonlinear estimates, for any $T>0$ we get
\begin{align*}
\norm{w}_{F^{n/2}\cap Z^{n/2}}\les \norm{\phi}_{\dot
B^{n/2}_{2,1}}+\e T \norm{u_\e}_{L_t^\infty \dot
B^{(n+4)/2}_{2,1}}+\delta^2\norm{w}_{F^{n/2}\cap
Z^{n/2}}+\e\norm{u_\e}^3_{F^{n/2}\cap Z^{n/2}}.
\end{align*}
Then we get by \eqref{eq:pers}
\begin{align}\label{eq:limitHigh}
\norm{w}_{F^{n/2}\cap Z^{n/2}}\les &\norm{\phi}_{\dot
B^{n/2}_{2,1}}+\e T \norm{\phi_1}_{\dot
B^{(n+4)/2}_{2,1}}+\e\delta^3.
\end{align}
Now we assume $\phi_1=\phi_2=\varphi\in \dot B^{n/2}_{2,1}$ with
small norm. For fixed $T>0$, we need to prove that $\forall\
\eta>0$, there exists $\sigma>0$ such that if $0<\e<\sigma$ then
\begin{align}\label{eq:limitHs}
\norm{S_T^\e(\varphi)-S_T(\varphi)}_{C([0,T];\dot
B^{n/2}_{2,1})}<\eta
\end{align}
where $S_T^\e$ is the solution map corresponding to \eqref{eq:dGLw}
and $S_T=S_T^0$. We denote $\varphi_K=P_{\leq K}\varphi$. Then we
get
\begin{align*}
&\norm{S_T^\e(\varphi)-S_T(\varphi)}_{C([0,T];\dot
B^{n/2}_{2,1})}\nonumber\\
\leq&\norm{S_T^\e(\varphi)-S_T^\e(\varphi_K)}_{C([0,T];\dot
B^{n/2}_{2,1})}\nonumber\\
&+\norm{S_T^\e(\varphi_K)-S_T(\varphi_K)}_{C([0,T];\dot
B^{n/2}_{2,1})}+\norm{S_T(\varphi_K)-S_T(\varphi)}_{C([0,T];\dot
B^{n/2}_{2,1})}.
\end{align*}
From the uniform global well-posedness and \eqref{eq:limitHigh}, we
get
\begin{align}
\norm{S_T^\epsilon(\varphi)-S_T(\varphi)}_{C([0,T];\dot
B^{n/2}_{2,1})}\les \norm{\varphi_K-\varphi}_{\dot B^{n/2}_{2,1}}+\e
C(T,K, \norm{\varphi}_{\dot B^{n/2}_{2,1}}).
\end{align}
We first fix $K$ large enough, then let $\e$ go to zero, therefore
\eqref{eq:limitHs} holds.

\subsection*{Acknowledgment}
Z. Guo is supported in part by NNSF of China (No.11371037), and
C. Huang is supported in part by NNSF of China (No. 11201498).

\end{document}